\documentclass[12pt]{amsart}
\usepackage[all]{xy}
\usepackage{pdfsync}
\usepackage{listings}
\usepackage[utf8]{inputenc}
\usepackage[T1]{fontenc}
\usepackage{amsthm,amsmath,amssymb}
\usepackage[backend=bibtex,style=numeric,url=false,doi=false]{biblatex}
\bibliography{eigene,andere,arxiv,mathscinet}
\usepackage{cleveref}
 \newtheorem{Thm}{Theorem}[section]
 \newtheorem{Lem}[Thm]{Lemma}
 \newtheorem{Prop}[Thm]{Proposition}
 \newtheorem{Cor}[Thm]{Corollary}
 \newtheorem{Conj}[Thm]{Conjecture}
\theoremstyle{remark}
 \newtheorem{Expl}[Thm]{Example}

 \newtheorem{Rem}[Thm]{Remark}
\theoremstyle{definition}
 \newtheorem{Def}[Thm]{Definition}
\numberwithin{equation}{section}
\newcommand\N{\mathbb N}
\newcommand\Z{\mathbb Z}

\newcommand\braid{c}
\newcommand\tr{\operatorname{tr}}
\newcommand\End{\operatorname{End}}

\newcommand\adhit{\triangleright}
\newcommand\inv{^{-1}}
\def\HM#1.#2.#3.#4.{{^{#1}_{#3}\mathcal M^{#2}_{#4}}}
\newcommand\lYD[1]{{^{#1}_{#1}\mathcal{YD}}}

\newcommand\id{\operatorname{id}}

\newcommand\ot{\otimes}
\newcommand{\ou}[1]{\underset{{#1}}{\otimes}}
\newcommand{\co}[1]{\mathrel{\mathop{\Box}_{#1}}}

\newcommand\Ind{\operatorname{Ind}}
\newcommand\Stab{\operatorname{Stab}}
\newcommand\Rep{\operatorname{Rep}}
\newcommand\Irr{\operatorname{Irr}}
\newcommand\Vect{\operatorname{Vect}}
\newcommand\Tr{\operatorname{Tr}}
\newcommand\lquot{\backslash}

\newcommand\Hom{\operatorname{Hom}}

\newcommand\Indtobim{\mathcal F}
\newcommand\IndtoYD{\mathcal G}

\newcommand\Q{\mathbb Q}
\newcommand\CC{\mathbb C}

\newcommand\C{\mathcal C}

\newcommand\galaut{\sigma}
\newcommand\adams{\psi}
\newcommand\auta{\psi}
\newcommand\autb{\tilde\psi}
\newcommand\doubleadams{\hat\psi}

\newcommand\semidir\rtimes

\begin{document}
\title[Quasitensor autoequivalences]{Some quasitensor autoequivalences of Drinfeld doubles of finite groups}
\author{Peter Schauenburg}
\address{Institut de Math{\'e}matiques de Bourgogne, UMR 5584 du CNRS
\\
Universit{\'e} de Bourgogne\\
Facult{\'e} des Sciences Mirande\\
9 avenue Alain Savary\\
BP 47870 21078 Dijon Cedex\\
France
}
\email{peter.schauenburg@u-bourgogne.fr}
\subjclass{18D10,16T05,20C15}

\keywords{}
\thanks{Research partially supported through a FABER Grant by the \emph{Conseil régional de Bourgogne}}
\begin{abstract}We report on two classes of autoequivalences of the category of Yetter-Drinfeld modules over a finite group, or, equivalently the Drinfeld center of the category of representations of a finite group. Both operations are related to the $r$-th power operation, with $r$ relatively prime to the exponent of the group. One is defined more generally for the group-theoretical fusion category defined by a finite group and an arbitrary subgroup, while the other seems particular to the case of Yetter-Drinfeld modules.  Both autoequivalences preserve higher Frobenius-Schur indicators up to Galois conjugation, and they preserve tensor products, although neither of them can in general be endowed with the structure of a monoidal functor.
\end{abstract}
\maketitle

\section{Introduction}
\label{sec:introduction}

Let $G$ be a finite group of exponent $e$. The complex characters of $G$ take values in the cyclotomic field $\Q(\zeta_e)$, where $\zeta_e$ is a primitive $e$-th root of unity. The Galois conjugates of a character $\chi$, i.~e.\ the images of $\chi$ under the operation of the Galois group of $\Q(\zeta_e)$, are also characters of $G$. This gives an operation of the group of prime residues modulo $e$ on the characters of $G$, or an operation by autoequivalences on the category $\Rep(G)$. Another way of describing this operation is as a  special case of an Adams operator, mapping $\chi$ to the character $\adams_r(\chi)$ characterized by $\adams_r(\chi)(g)=\chi(g^r)$ for $r$ relatively prime to $e$. (We admit that the reference to Adams operators is something of an exaggeration, since the latter are defined for any $r$, and the case where $r$ is relatively prime to $e$ is only a very particular one.) We note that Galois conjugation of characters preserves higher Frobenius-Schur indicators and tensor products. However, as we learnded from Alexei Davydov, it does not in general, give rise to a \emph{monoidal} autoequivalence of $\Rep(G)$. Thus, Galois conjugation defines what is sometimes called a quasitensor equivalence of $\Rep(G)$, in particular an automorphism of the Grothendieck ring of $\Rep(G)$.

We will study operations with somewhat similar behavior, also related to taking $r$-th powers (with $r$ relatively prime to $e$), on the Drinfeld center of $\Rep(G)$, i.~e.\ the category of Yetter-Drinfeld modules $\lYD G$, which can also be described as the module category of the Drinfeld double $D(G)$.

The first and rather more obvious operation, treated in \cref{sec:an-auto-grad}, is defined more generally on the category $\HM\Gamma..G.G.$ of $\CC G$-bimodules graded by a finite group $\Gamma$ containing $G$ as a subgroup; the case of Yetter-Drinfeld modules is recovered by choosing $\Gamma=G\times G$ with $G$ embedded diagonally. Simple objects in $\HM\Gamma..G.G.$ are described ``locally'' by representations of certain stabilizer subgroups of the group $G$. The group of prime residues modulo $e$ can thus be made to act on $\HM\Gamma..G.G.$ by Galois conjugating the local characters of the stabilizer subgroups. By the formula for higher Frobenius-Schur indicators obtained recently in \cite{Sch:CHFSIFCCIFG} (and much earlier for the case of Drinfeld doubles in \cite{KasSomZhu:HFSI}) it is obvious that the effect of this operation on the higher Frobenius-Schur indicators of an object is to Galois conjugate them as well. (To compare with the special case of group representations, note that higher Frobenius-Schur in the latter case are integers.) It is less obvious, though perhaps not entirely surprising, that the Galois conjugation operation also preserves tensor products in $\HM\Gamma..G.G.$. 

The second operation, which we will introduce in \cref{sec:an-auto-yett}, is particular to the categories $\lYD G$ of Yetter-Drinfeld modules. It is defined simply by raising the degrees of all the homogeneous elements in a Yetter-Drinfeld modules to the $r$-th power. If we describe simple Yetter-Drinfeld modules ``locally'' by irreducible representations of the centralizers $C_G(g)$ of elements $g\in G$, then the operation does not change the representation of $C_G(g)$ at all, but views it as associated to the element $g^r$ instead of $g$. While both descriptions are very elementary, we could not see any elementary reason why the operation should behave as it does on higher Frobenius-Schur indicators (it preserves them up to Galois conjugation) or tensor products (it preserves them). We will use the Verlinde formula for multiplicities in the modular tensor category $\lYD G$, and the Bantay-type formula for indicators to prove these facts.

We note that the two operations on $\lYD G$ thus associated in different ways to the $r$-th power operation in $G$ do not usually coincide. Both give quasitensor autoequivalences that preserve higher Frobenius-Schur indicators up to Galois conjugation. Taking a suitable ``quotient'' of the two types of operations yields a quasitensor autoequivalence that preserves higher Frobenius-Schur indicators (instead of Galois conjugating them). However, even this ``better-behaved'' operation cannot be endowed, in general, with the structure of a monoidal functor. We regard it briefly in \cref{sec:an-auto-pres}. The fact that it conserves Frobenius-Schur indicators entails, \cref{Cor:counting-formula}, a certain combinatorial counting identity  for finite groups, based on the counting formulas in \cite{KasSomZhu:HFSI} involving higher Frobenius-Schur indicators.

The author's search for autoequivalences of $\lYD G$ was at first motivated by certain computational data: Keilberg \cite{MR2941565,MR2925444,MR3178056} and Courter \cite{MR3103664} have accumulated rich results on the higher Frobenius-Schur indicators of irreducible representations of the Drinfeld doubles of certain finite groups. Both for the groups ---semidirect products of cyclic groups with cyclic or dihedral groups---  treated by Keilberg, and for the symmetric groups treated by Courter (with heavy computer calculations) it turns out that many different representations can have identical sets of indicators (Courter terms these representations I-equivalent), without any automorphisms of the Drinfeld doubles at hand that would map the I-equivalent representations to each other. With respect to this phenomenon, the autoequivalences studied in the present paper are both a success (they explain why certain representations of the Drinfeld doubles are I-equivalent) and a failure (the larger part of I-equivalences between different representations remains unexplained).

In~\cref{sec:examples} we will look at a few small examples of the two types of operations on simple Yetter-Drinfeld modules, and in \cref{sec:doubl-symm-groups} we will discuss how the second type of operation fails utterly to help explain most of the I-equivalent objects in $\lYD{S_n}$ found by Courter: It seems that on $\lYD{S_n}$ the two generally different $r$-th power operations are identical, a fact that we cannot, for the moment, prove, but that we tested using the GAP software up to \newcommand\bishierher{$S_{23}$}\bishierher.

\section{Preliminaries}
\label{sec:preliminaries}

Let $e\in\N$ and denote by $\mathbb Q(\sqrt[e]1)$ the cyclotomic field obtained by adjoining a primitive $e$-th root of unity $\zeta_e$ to the rationals. For $r$ relatively prime to $e$ we denote by $\galaut_r$ the automorphism of $\mathbb Q(\sqrt[e]1)$ mapping $\zeta_e$ to $\zeta_e^r$.

Let $G$ be a finite group, $e=\exp(G)$, and $\chi$ a complex character of $G$. We denote by $\adams_r(\chi)$ the character defined by $\adams_r(\chi)(g)=\chi(g^r)$ for $g\in G$. The assignment $\psi_r$ is known as the $r$-th Adams operator; for general $r$ it gives a virtual character of $G$, but we will only be concerned with the case (somewhat trivial in the theory of Adams operators, cf.\ \cite[§12B]{MR632548}) that $r$ is relatively prime to $e$. The well-known fact that $\adams_r(\chi)$ is a character in this case follows from Brauer's theorem: In fact $\chi$ is the character of a representation $W$ which is defined over $\mathbb Q(\sqrt[e]1)$; that is, with respect to a suitable basis of $W$ the action of $g\in G$ is given by a matrix with entries in $\mathbb Q(\sqrt[e]1)$. Applying the map $\galaut_r$ to all the entries of all these matrices yields a conjugated representation, which we will denote by $\adams_r(W)$, and whose character is $\adams_r(\chi)$.

It is well-known (and trivial to check) that $\adams_r(V\ot W)\cong\adams_r(V)\ot\adams_r(W)$ and $\nu_m(\adams_r(V))=\nu_m(V)$ for all $V,W\in\Rep(G)$; the latter because higher Frobenius-Schur indicators of group representations are integers.

We will exhibit operators of somewhat similar behavior defined on certain group-theoretical fusion categories; to wit: One type of operator is defined on the category of graded bimodules $\HM\Gamma..G.G.$ with $\Gamma$ a finite group and $G\subset\Gamma$ a subgroup. The other type is defined on the category of Yetter-Drinfeld modules $\lYD G$, equivalent in fact to $\HM\Gamma..G.G.$ when we choose $G$ embedded diagonally in $\Gamma=G\times G$. The simple objects of $\HM\Gamma..G.G.$ can be described in terms of group representations. We need this description (which can be extracted from \cite{Zhu:HARRHA} and \cite{Sch:HAEMC}) very explicitly.

The category $\HM \Gamma..G.G.$ is defined as the category of $\CC G$-bimodules over the group algebra of $G$, considered as an algebra in the category of $\CC \Gamma$-comodules, that is, of $\Gamma$-graded vector spaces. Thus, an object of $\HM \Gamma..G.G.$ is a $\Gamma$-graded vector space $M\in\Vect_\Gamma$ with a two-sided $G$-action compatible with the grading in the sense that $|gmh|=g|m|h$ for $g,h\in G$ and $m\in M$.

The category $\HM \Gamma..G.G.$ is a fusion category. The tensor product is the tensor product of $\CC G$-bimodules. Simple objects are parametrized by irreducible representations of the stabilizers of right cosets of $G$ in $\Gamma$. More precisely, let $D\in G\lquot \Gamma/G$ be a double coset of $G$ in $\Gamma$, let $d\in D$, and let $S=\Stab_H(dG)=G\cap (d\adhit G)$ be the stabilizer in $G$ of the right coset $dG$ under the action of $G$ on its right cosets in $\Gamma$. Then the subcategory $\HM D..G.G.\subset \HM \Gamma..G.G.$ defined to contain those objects the degrees of all of whose homogeneous elements lie in $D$ is equivalent to the category $\Rep(S)$ of representations of $S$. The equivalence $\HM D..G.G.\to\Rep(S)$ takes $M$ to $(M_{dG})/G\cong(M/G)_{dG/G}$, the space of those vectors in the quotient of $M$ by the right action of $G$ whose degree lies in the right coset of $d$. Details are in \cite{Zhu:HARRHA,Sch:HAEMC}. We will denote the inverse equivalence by $\Indtobim_d\colon\Rep(S)\to\HM GdG..G.G.$. Explicitly, it can be described as follows: For $V\in\Rep(S)$ let $N:=\Ind_S^G=\CC G\ou{\CC S}V$ be the induced representation of $G$, which we endow with a $\Gamma/G$-grading by setting $|g\ot v|=gdG$; this grading is obviously compatible with the left action of $G$ so that (with a hopefully obvious definition) $N\in\HM \Gamma/G..G..$. Now we have
\begin{align}
  \label{eq:1}
  \Indtobim_d(V)&=N\co{\CC [\Gamma/G]}\CC G\\\notag&=\left\{\left.\sum_{\gamma\in\Gamma} n_\gamma\ot \gamma\in N\ot \CC G\right|n_\gamma\in N_{\gamma G}\right\}
    \\\notag&=\bigoplus_{\gamma}N_{\gamma G}\ot \CC \gamma G,
\end{align}
with the last sum running over a set of representatives of the right cosets of $G$ in $\Gamma$, the grading $|n\ot\gamma|=\gamma$, and the bimodule structure $g(n\ot \gamma)h=gn\ot g\gamma h$.

The functors $\Indtobim_d$ combine into a category equvialence
\begin{equation*}
  \bigoplus_d\Rep(\Stab_G(dG))\xrightarrow{(\Indtobim_d)_d}\HM \Gamma..G.G.
\end{equation*}
in which the sum runs over a set of representatives of the double cosets of $G$ in $\Gamma$.

The category $\lYD G=\lYD{\CC G}$ of (left-left) Yetter-Drinfeld modules over $\CC G$ has objects the $G$-graded vector spaces with a left $G$-action compatible with the grading in the sense that $|gv|=g|v|g\inv$ for $g\in G$ and $v\in V\in\lYD G$. The category $\lYD G$ is the (right) center of the category $\HM G....$ of $G$-graded vector spaces: The half-braiding $c\colon U\ot V\to V\ot U$ between a graded vector space $U$ and a Yetter-Drinfeld module $V$ is given by $u\ot v\mapsto |u|v\ot u$. 

 Simple objects of $\lYD G$ are parametrized by irreducible representations of the centralizers in $G$ of elements of $G$. (In fact this can be viewed as a special case of the description of graded bimodules above.). More precisely, let $g\in G$ and let $C_G(g)$ be the centralizer of $g$ in $G$. Then a functor 
 \begin{equation*}
   \IndtoYD_g\colon\Rep(C_G(g))\to \lYD G
 \end{equation*}
can be defined by sending $V\in\Rep(C_G(g))$ to the induced $\CC G$-module $\Ind_{C_G(g)}^GV=\CC G\ou{\CC C_G(g)}V$ endowed with the grading given by $|x\ot v|=xgx\inv$ for $x\in G$ and $v\in V$. In this way we obtain a category equivalence
\begin{equation*}
  \oplus_g\Rep(C_G(g))\xrightarrow{(\IndtoYD_g)_g}\lYD G.
\end{equation*}
The sum runs over a set of representatives of the conjugacy classes of $G$, and the image of the functor $\IndtoYD_g$ consists of those Yetter-Drinfeld modules the degrees of whose homogeneous elements lie in the conjugacy class of $g$. The images of the functors $\IndtoYD_g$ for conjugate elements coincide. More precisely, we have
\begin{equation*}
  \IndtoYD_{x\adhit g}(x\adhit W)\cong\IndtoYD_g(W)
\end{equation*}
for $W\in\Rep(C_G(g))$, with $x\adhit W\in \Rep(x\adhit C_G(g))=\Rep(C_G(x\adhit g))$.

The categories $\HM\Gamma..G.G.$ considered in this paper are examples of fusion categories \cite{MR2183279}. As we have already mentioned several times, we will be concerned with certain structure-preserving autoequivalences of such categories. To fix the notions, a monoidal functor or tensor functor $\mathcal F\colon\mathcal C\to\mathcal D$ consists of a functor $\mathcal F$ together with a natural isomorphism $\xi\colon \mathcal F(X)\otimes \mathcal F(Y)\to\mathcal F(X\ot Y)$ which is required to fulfill a compatibility condition with the associativity constraints of $\mathcal C$ and $\mathcal D$. If this compatibility condition is \emph{not} required, we shall speak of a quasitensor functor.

We will write $\langle M,N\rangle:=\dim_\CC (\Hom_{\C}(M,N))$ for objects $M,N$ in a semisimple category.

\section{An autoequivalence of graded bimodules}
\label{sec:an-auto-grad}

In this section we define an autoequivalence of the category of graded bimodules $\HM \Gamma..G.G.$ that generalizes the $r$-th Adams operator for $r$ prime to the exponent of $G$ in case $\Gamma=G$. (Note that $\HM G..G.G.\cong\Rep(G)$ as tensor categories.)

Let $\Gamma$ be a finite group and $G\subset \Gamma$ a subgroup. Let $e=\exp(G)$, and $r$ an integer relatively prime to $e$. Since $r$ is also relatively prime to the exponent of all the stabilizer subgroups of $G$ that occur in the structure description of $\HM \Gamma..G.G.$ given above, the characters of all those stabilizer subgroups take values in a cyclotomic field of order prime to $r$. Thus, we can apply the automorphism $\adams_r$ to those characters.

\begin{Def}\label{def:sigmafunctor}
  Let $\Gamma$ be a finite group, $G\subset \Gamma$ a subgroup, and $e=\exp(G)$. Let $r$ be an integer relatively prime to $e$. Then we define an autoequivalence $\auta_r$ of $\HM \Gamma..G.G.$ by
  \begin{equation}
    \label{eq:2}
   \auta_r\left(\Indtobim_d(\eta)\right):=\Indtobim_d(\adams_r(\eta)).
  \end{equation}
  for $d\in \Gamma$.
\end{Def}

\begin{Lem}\label{Lem:arith-bases}
  One can choose bases $B=(b_i)$ for each $M\in\HM \Gamma..G.G.$ consisting of $\Gamma$-homogeneous elements such that, if $L_i^j(M,g),R_i^j(M,g)\in \C$ for $g\in G$ are defined by $gb_i=\sum L_i^j(M,g)b_j$ and $b_ig=\sum R_i^j(M,g)b_j$, then
  \begin{enumerate}
  \item $L_i^j(M,g)\in \Q(\zeta_e)$.
  \item $(R_i^j(M,g))_{i,j}$ is a permutation matrix.
  \item $L_i^j(\auta_r(M),g)=\galaut_r(L_i^j(M,g))$.
  \item $R_i^j(\auta_r(M),g)=R_i^j(M,g)$.
  \end{enumerate}
\end{Lem}
\begin{proof}
  It suffices to pick such bases for each $\Indtobim_d(W)$, where $d$ runs through a system of representatives for the double cosets of $G$ in $\Gamma$, and $W$ runs through the irreducible representations of $S=\Stab_G(dG)$.

  By Brauer's Theorem, we can pick a basis $(b_\alpha')$ for such a representation $W\in\Rep(S)$ such that the corresponding matrix representation of $S$ has coefficients in $\Q(\zeta_e)$. The representation $\adams_r(W)$ can then be realized on the same underlying vector space by Galois conjugating those matrix coefficients. We will choose to realize all the representations in the same $\adams_r$-orbit in this way. For the induced representation $N=\Ind_S^G(W)=\CC G\ou{\CC S}W$ we choose the basis $b''_{c,\alpha}=c\ot b'_\alpha$ indexed by the index set of the basis $(b'_\alpha)$ along with a set of representatives of $G/S$. Note that $b''_{c,\alpha}$ has $\Gamma/G$-degree $cdG$. It is easy to check that the matrix of the left action of $g\in G$ on $\Ind_S^G(W)$ has coefficients in $\Q(\zeta_e)$, and that the $\galaut_r$-conjugate matrix describes the action on $\Ind_S^G(\adams_r(W))$. Let $\tilde N=N\ot \CC \Gamma$, with the diagonal left $G$-action and the right $G$-action on $\CC \Gamma$. Then $\tilde N$ has basis $(b''_{c,\alpha}\ot\gamma)$, indexed by $c,\alpha$ as above and $\gamma\in G$; with respect to this basis, the right action of $G$ is by permutation matrices, and the left action is by the Kronecker products of the left action on $N$ and permutation matrices. According to  \eqref{eq:1}, a subset of the chosen basis of $\tilde N$ gives a basis of $\Indtobim_d(W)\in\HM \Gamma..G.G.$, and thus the matrices of the left and right action of $G$ on the latter behave as claimed.
\end{proof}

\begin{Prop}
  Let $\Gamma$ be a finite group, $G\subset \Gamma$ a subgroup, $e=\exp(G)$ and $r$ relatively prime to $e$.
  \begin{enumerate}
  \item The functor $\auta_r$ from \cref{def:sigmafunctor} is a quasitensor autoequivalence.
  \item We have $\nu_m(\auta_r(M))=\galaut_r(\nu_m(M))$ for each $M\in\HM \Gamma..G.G.$.
  \end{enumerate}
\end{Prop}
\begin{proof}
  The second claim is obvious from the indicator formula obtained in \cite{Sch:CHFSIFCCIFG}. For the first claim we have to check that the multiplicity
  \begin{equation}\label{eq:3}
    \langle M\ou{\CC G}N,P\rangle=\dim\Hom_{G-G}^\Gamma(M\ou{\CC G}N,P)
  \end{equation}
  for three simple objects $M,N,P\in\HM \Gamma..G.G.$ does not change upon application of $\auta_r$ to all three objects. We choose bases $(b_i)$ for $M$, $(c_i)$ for $N$ and $d_i$ for $P$ for the three objects as in \cref{Lem:arith-bases}. The morphisms we ``count'' in \eqref{eq:3} are described by $\CC $-linear maps $f\colon M\ot N\to P$ subject to the conditions $gf(m\ot n)=f(gm\ot n)$, $f(mg\ot n)=f(m\ot gn)$ and $f(m\ot ng)=f(m\ot n)g$ for all $g\in G$, where it suffices to let $m,n$ run through bases of $M$ and $N$. Write $f(b_i\ot c_j)=\sum \phi_{ij}^kd_k$ for $\phi_{ij}^k\in \CC $. Then the three conditions on $f$ translate to relations between the matrix elements $\phi_{ij}^k$ and the matrix elements for the left and right $G$-actions from \cref{Lem:arith-bases}, to wit:
  \begin{align*}
    \sum L_i^k(M,g) \phi_{kj}^\ell&=\sum \phi_{ij}^mL_m^\ell(P,g)\\
    \sum R_i^k(M,g)\phi_{kj}^\ell&=\sum \phi_{ik}^\ell L_j^k(N,g)\\
    \sum R_j^k(N,g)\phi_{ik}^\ell&=\sum\phi_{ij}^kR_k^\ell(P,g)
  \end{align*}
  where the respective sums run over the index occurring twice. When we apply the functor $\auta_r$ to the three objects $M,N,P$, the matrices of the left actions are conjugated by $\galaut_r$, as are the (integer) matrices of the right actions. Thus, conjugating the matrix $\phi$ of $f$ by $\galaut_r$ as well yields a homomorphism between the conjugated objects, whence the claim.
\end{proof}

\section{An autoequivalence of Yetter-Drinfeld modules}
\label{sec:an-auto-yett}

Let $G$ be a finite group, and $r$ an integer relatively prime to $\exp(G)$. We consider the following autoequivalence $\autb_r$ of the category $\lYD G$ of Yetter-Drinfeld modules: For $M\in\lYD G$ let $\autb_r(M)=M$ as left $\CC G$-module, but modify the grading by setting $\autb_r(M)_{g^r}=M_g$. We will denote by $\autb_r(m)$ the element $m\in M$ considered as an element of $\autb_r(M)$, so that $|\autb_r(m)|=|m|^r$ if $m$ is homogeneous. The fact that $\autb_r(M)$ is a Yetter-Drinfeld module is a trivial consequence of the power map's being equivariant with respect to the adjoint action. $\autb_r$ is an autoequivalence since $r$ is relatively prime to $\exp(G)$, the inverse being $\autb_s$ when $rs\equiv 1\mod \exp(G)$.

We will denote by $(X_i)_{i\in I}$ a representative set for the isomorphism classes of simples of $\lYD G$, and define $\autb_r(i)$ by $\autb_r(X_i)=X_{\autb_r(i)}$ for $i\in I$.

We need to examine the behavior of the modular structure of the category $\lYD G$ under the functor $\autb_r$. This is rather simple for the $T$-matrix or ribbon structure $\theta$ of the category: The map $\theta_M$ maps $m$ of degree $g$ to $gm$ of the same degree. Therefore, $\theta_{\autb_r(M)}=(\theta_M)^r$. If $M=X_i$ is simple, so $\theta_M=\omega_i\id_M$ for a root of unity $\omega_i$, we obtain $\omega_{\autb_r(i)}=(\omega_i)^r=\galaut_r(\omega_i)$.

We proceed to investigate the $S$-matrix, denoting by $\braid$ the braiding in the braided monoidal category $\lYD G$. Let $M,N\in\lYD G$. Then for $m\in M_g$ and $n\in N_h$ we have
\begin{align*}
  \braid^2(m\ot n)&=\braid(gn\ot m)\\
  &=|gn|m\ot gn\\
  &=ghg\inv m\ot gn.
\end{align*}
We can endow $M\ot N$ with a $G\times G$-grading composed of the $G$-gradings of $M$ and $N$. Then
\begin{align*}
  \deg_{G\times G}\braid^2(m\ot n)&=(ghg\inv g(ghg\inv)\inv,ghg\inv)\\
  &=(ghg^2h\inv g\inv,ghg\inv).
\end{align*}

For a finite group $\Gamma$, a $\Gamma$-graded vector space $V$, and an endomorphism $f$ of $V$ write $f_0$ for the trivial component of $f$ with respect to the $\Gamma$-grading of $\End(V)$. Then $\Tr(f)=\Tr(f_0)$.

In our example, considering the $G\times G$-grading of $M\ot N$, we see that $\braid^2(m\ot n)$ has the same degree as $m\ot n$ if and only if $g$ and $h$ commute; thus $(\braid^2)_0(m\ot n)=0$ unless $(g,h)=1$, and $(\braid^2)_0(m\ot n)=hm\ot gn$ if $(g,h)=1$.

The last equation implies also that $(\braid^2)_0$ is of finite order dividing $\exp(G)$, and thus, for $r$ relatively prime to $\exp(G)$, we have
$\tr\left((\braid^2)_0^r\right)=\galaut_r\left(\tr\left((\braid^2)_0\right)\right)$. On the other hand, $(\braid^2_{\autb_r(M),\autb_r(N)})_0=\left(\braid^2_{MN}\right)_0^r$ (note that $gh=hg$ if and only if $g^rh^r=h^rg^r$). Taking these results together, we obtain
$\tr(\braid^2_{\autb_rM,\autb_rN})=\galaut_r(\tr(\braid^2_{MN}))$.

\begin{Thm}
  Let $G$ be a finite group, and $r$ an integer relatively prime to $\exp(G)$.

  Then $\autb_r\colon\lYD G\to\lYD G$ is a quasi-tensor autoequivalence of $\lYD G$.
  
  We have $S_{\autb_r(i),\autb_r(j)}=\galaut_r(S_{ij})$ for simples $i,j$ of $\lYD G$, and $\nu_m(\autb_r(V))=\galaut_r(\nu_m(V))=\nu_{mr}(V)$ for $m\in\Z$.
\end{Thm}
\begin{proof}
  We have already proved that $\autb_r$ preserves the $S$-matrix up to Galois conjugation.

  By the Verlinde formula (see \cite[Thm.~3.1.13]{BakKir:LTCMF}), this implies that the multiplicities $N_{ij}^k$ of a simple $X_k$ in $X_i\ot X_j$ are also preserved up to Galois conjugation, that is, $N_{\autb_r(i),\autb_r(j)}^{\autb_r(k)}=\galaut_r(N_{ij}^k)$. But the multiplicities are integers, and thus $N_{\autb_r(i),\autb_r(j)}^{\autb_r(k)}=N_{ij}^k$.

Now $\autb_r$ also preserves the $T$-matrix up to Galois conjugation. Thus the formula from \cite[Thm.~7.5]{NgSch:FSIESC}, generalizing Bantay's formula for degree two Frobenius-Schur indicators to higher indicators, implies that Frobenius-Schur indicators behave under $\autb_r$ as claimed above.
  \end{proof}

\begin{Rem}
  Iovanov, Mason, and Montgomery \cite{2012arXiv1208.4153I} have found examples of finite groups $G$ such that not all the higher Frobenius-Schur indicators of simple modules of the Drinfeld double $D(G)$ are integers. For such $G$ (the simplest example in \cite{2012arXiv1208.4153I} has order $5^6$), the autoequivalences $\autb_r$ do not preserve Frobenius-Schur indicators, and thus they cannot be endowed with the structure of a \emph{monoidal} autoequivalence of $\lYD G$.
\end{Rem}

It remains to observe how the functors $\autb_r$ behave with respect to the ``local'' description of simples of $\lYD G$ in terms of representations of centralizers. The proof of the following result should be rather obvious:
\begin{Lem}
  Let $G$ be a finite group, and $r$ an integer relatively prime to $\exp(G)$. For $g\in G$ and $V\in\Rep(C_G(g))$ we have
  \begin{equation}
    \label{eq:4}
    \autb_r\left(\IndtoYD_g\left(V\right)\right)=\IndtoYD_{g^r}\left(V\right)
  \end{equation}
\end{Lem}
In other words, simple Yetter-Drinfeld modules are parametrized by an element $g$ of $G$ and an irreducible representation of its centralizer. Applying $\autb_r$ amounts to taking the same representation, but regarding it as belonging to $g^r$ (which has the same centralizer) instead of $g$.

Typically, doing explicit calculations with the category $\lYD G$, one would choose a system $\mathfrak C$ of representatives for the conjugacy classes in $G$, so that every simple is of the form $\IndtoYD_g(V)$ for a \emph{unique} $g\in \mathfrak C$ and $V\in\Irr(C_G(g))$. The preceding lemma, by contrast, does not follow this partition of the simples, since $g\in\mathfrak C$ does not imply $g^r\in\mathfrak C$.
\begin{Rem}
  Let $G$ be a finite group, $\mathfrak C$ a system of representatives of the conjugacy classes of $G$, and $r$ an integer relatively prime to $\exp(G)$. For $g\in\mathfrak C$ choose $x\in G$ such that $x\adhit g^r\in\mathfrak C$. Then for $V\in\Rep(C_G(g))$ we have
  \begin{equation}
    \label{eq:5}
    \autb_r(\IndtoYD_g(V))=\IndtoYD_{x\adhit g^r}(x\adhit V).
  \end{equation}
\end{Rem}

\section{An autoequivalence preserving indicators}
\label{sec:an-auto-pres}

In the previous section we have constructed a quasitensor autoequivalence of $\lYD G$ preserving higher Frobenius-Schur indicators up to Galois conjugation. Another such equivalence is given by the results in \cref{sec:an-auto-grad}, since the category of Yetter-Drinfeld modules can be viewed as a special case of the category of graded bimodules. Combining the two (rather, one of them with the inverse of the other) we obtain an quasitensor autoequivalence of $\lYD G$ that preserves higher Frobenius-Schur indicators. It is still not a \emph{monoidal} equivalence by recent results of Davydov.

\begin{Def}
  Let $G$ be a finite group and $r$ an integer relatively prime to the exponent $e=\exp(G)$. Define an autoequivalence $\doubleadams_r$ of $\lYD G$ by $\doubleadams_r=\autb_s\circ\auta_r$, where $sr\equiv 1\mod e$.

  Thus, $\doubleadams_r\left(\IndtoYD_{g^r}(V)\right)=\IndtoYD_g(\adams_r(V))$.
\end{Def}
\begin{Cor}
  $\doubleadams_r$ is a quasitensor autoequivalence of $\lYD G$ and satisfies $\nu_m(\doubleadams_r(V))=\nu_m(V)$ for all $V\in\lYD G$.
\end{Cor}

\begin{Rem}
  $\doubleadams_r$ cannot, in general, be endowed with the structure of a monoidal autoequivalence of $\lYD G$. In fact the restriction of $\doubleadams_r$ to the monoidal subcategory $\Rep(G)\subset\lYD G$ of those Yetter-Drinfeld modules whose grading is trivial is the usual Adams operator. There are examples where these Adams operators with $r\equiv{-1}\mod e$ do not arise from a monoidal autoequivalence. The following argument is due to Alexei Davydov: First, for a simple group $G$, any autoequivalence of $\Rep(G)$ must arise from a group automorphism of $G$. In fact, autoequivalences of $\Rep(G)$ are classified by $G$-Bi-Galois algebras \cite{Sch:HBE,Dav:GAMFBCRFG}. But if $G$ is simple, then by \cite[Thm.3.8 and Prop.6.1]{Dav:GAMFBCRFG}  such $G$-Bi-Galois algebras are trivial as one-sided $G$-Galois algebras, and thus they are described by group automorphisms (cf.\ \cite[Lem.3.11]{Sch:HBE}). Now for an automorphism of $G$ to give rise to the Adams operator $\adams_{-1}$, it has to be a class-inverting automorphism (see \cite{2014arXiv1412.8505D}), and there exist simple groups where such automorphisms do not exist (\cite{2014arXiv1412.8505D} gives the Mathieu group $M_{11}$.)  
\end{Rem}

\begin{Rem}
  Alexei Davydov pointed out that composing $\auta_r$ with $\autb_r$ (instead of using the inverse autoequivalence of one of the two) yields the Galois symmetry of $\lYD GG$. This can be read off from the description of the latter in \cite[p.689]{MR1770077}. Thus the two quasitensor autoequivalences can be viewed as decomposing the Galois symmetry (defined for any modular fusion category) in two factors, in the particular case of the Drinfeld double of a finite group.
\end{Rem}

\begin{Cor}\label{Cor:counting-formula}
  Let $G$ be a finite group, $r$ an integer prime to the exponent of $G$, $m\in\Z$, and $g,y\in G$. Then
  \begin{equation}
    \label{eq:6}
    |\{x\in G|x^m=(gx)^m=y\}|=|\{x\in G|x^{m}=(g^rx)^{m}=y^r\}|.
  \end{equation}
\end{Cor}
\begin{proof}
  According to \cite[Prop.~2.3]{2012arXiv1208.4153I} we have
  \begin{equation*}
     |\{x\in G|x^m=(gx)^m=y\}|=\sum_{\chi\in\Irr(C_G(g))}\nu_m\left(\IndtoYD_g(\chi)\right)\chi(y).
  \end{equation*}
  Since $\adams_r$ is a bijection on characters and  $\nu_m\circ\auta_r=\nu_m\circ\autb_r$,
  \begin{align*}
    \sum_{\chi}\nu_m\left(\IndtoYD_g(\chi)\right)\chi
  &=\sum_{\chi}\nu_m\left(\IndtoYD_g(\adams_r(\chi))\right)\adams_r(\chi)\\
      &=\sum_{\chi}\nu_m\left(\IndtoYD_{g^r}(\chi)\right)\adams_r(\chi),
  \end{align*}
  and
  \begin{align*}
\sum_{\chi}\nu_m\left(\IndtoYD_{g^r}(\chi)\right)\adams_r(\chi)(y)&=\sum_{\chi}\nu_m\left(\IndtoYD_{g^r}(\chi)\right)\chi(y^r)\\
      &=|\{x\in G|x^m=(g^rx)^m=y^r\}|
  \end{align*}
  finishes the proof
\end{proof}

\section{Some small Examples}
\label{sec:examples}

If $r=-1$, the Adams operator $\adams_{-1}$ is complex conjugation of characters, and $\auta_{-1}$ is complex conjugation of the ``local'' characters. It is known that higher Frobenius-Schur indicators of Drinfeld doubles are always real, so the operations $\auta_{-1}$ and $\autb_{-1}$ on $\lYD G$ preserve indicators, for any group $G$.
\begin{Expl}
  Consider the semidirect product
  \begin{equation*}
    G=\Z_3^2\semidir \Z_2
  \end{equation*}
  where $\Z_2$ acts on the first factor of $\Z_3^2$ by inversion, and trivially on the second. In other words
  \begin{equation*}
    G=\langle a,b,\tau|a^3=b^3=\tau^2=1,\tau a=a^2\tau,\tau b=b\tau\rangle.
  \end{equation*}
  The centralizer of $g=a$ is $C_G(a)=\langle a,b\rangle\cong C_3^2$. We denote its (linear) characters by $\chi_{uv}$ with $\chi_{uv}(a)=\zeta^u$ and $\chi_{uv}(b)=\zeta^v$, where $\zeta$ is a primitive third root of unity. By construction $a\inv$ is conjugate to $a$ by $\tau$. Note, though, that complex conjugation of $\chi_{uv}$ affects both its indices, while conjugation by $\tau$ only affects the first. The following diagram lists a dashed line between $\chi$ and $\chi'$ if $\IndtoYD_a(\chi')=\auta_{-1}(\IndtoYD_a(\chi))$ and a solid line if $\IndtoYD_a(\chi')=\autb_{-1}(\IndtoYD_a(\chi))$:
  \begin{equation*}    \xymatrix{\chi_{00}&\chi_{01}
      \ar@{--}[r]&\chi_{02}&\chi_{10}\ar@<.3ex>@{--}[r]\ar@<-.3ex>@{-}[r]&\chi_{20}&\chi_{11}\ar@{--}[r]\ar@{-}[d]&\chi_{22}\ar@{-}[d]\\
    &&&&&\chi_{21}\ar@{--}[r]&\chi_{12}}
  \end{equation*}
  Thus, the nine characters of $C_G(a)$ fall in four classes of necessarily I-equivalent characters.
\end{Expl}

\begin{Expl}
  Now consider $G=(C_3\times C_2^2)\semidir C_2$ with $C_2$ acting on $C_3$ by inversion, and on $C_2^2$ by switching factors. Thus
  \begin{equation*}
    G=\langle a,b,c,\tau|a^3=b^2=\tau^2=1,\tau a=a^2\tau,\tau b=c\tau,ab=ba,bc=cb\rangle.
  \end{equation*}
  We consider $g=a$ with $C_G(a)=\langle a,b,c\rangle\cong C_3\times C_2^2$, and label its characters $\chi_{uvw}$ with $\chi_{uvw}(a)=\zeta^u$, $\chi_{uvw}(b)=(-1)^v$, and $\chi_{uvw}(c)=(-1)^w$. Again, $a$ is conjugate to its inverse by $\tau$. Complex conjugation of $\chi_{uvw}$ affects only its first index, while conjugation by $\tau$ affects the first and switches the other two indices. Recording as in the previous example how $\auta_{-1}$ and $\autb_{-1}$ relate the various $\IndtoYD_a(\chi_{uvw})$ we obtain the following picture:
  
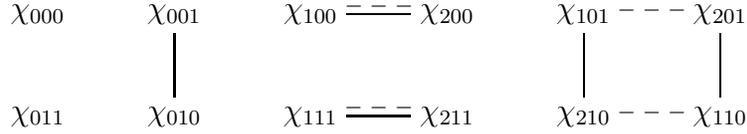
\begin{equation*}    \xymatrix{\chi_{000}&\chi_{001}\ar@{-}[d]&\chi_{100}\ar@<.3ex>@{--}[r]\ar@<-.3ex>@{-}[r]&\chi_{200}&\chi_{101}\ar@{--}[r]\ar@{-}[d]&\chi_{201}\ar@{-}[d]\\ \chi_{011}&\chi_{010}&\chi_{111}\ar@<.3ex>@{--}[r]\ar@<-.3ex>@{-}[r]&\chi_{211}&\chi_{210}\ar@{--}[r]&\chi_{110}}
  \end{equation*}
  In particular, the twelve simples of $\lYD G$ in the image of $\IndtoYD_a$ fall in six orbits under the actions of the two autoequivalences.
\end{Expl}

\begin{Expl}
  Consider $G=\langle a,b|a^{27}=b^3=1,ba=a^{19}b\rangle$, that is $G\cong C_{27}\semidir C_3$, with the action of the generator $b$ of $C_3$ sending the generator $a$ of $C_{27}$ to $a^{19}$.

  Consider $g=a^3b$. It is easy to check that $C_G(g)=\langle a^3,b\rangle\cong C_9\times C_3$. We label the characters of $C_G(g)$ as $\chi_{uv}$ with $\chi_{uv}(a^3)=\zeta^u$ and $\chi_{uv}(b)=\zeta^{3v}$, where $\zeta$ is a primitive ninth root of unity.

  We will study $\auta_2$ and $\autb_2$ (note that $2$ is a primitive root modulo $27$). Clearly $\auta_2$ acts on $\IndtoYD_g(\chi_{uv})$ by doubling both indices (modulo $9$ and $3$, respectively). On the other hand $\autb_2$ will map the image of $\IndtoYD_g$ to the image of $\IndtoYD_{g^2}$, and $g^2=a^6b^2$ is not conjugate to $g$. However, $g^4=a^{12}b$ is conjugate to $g$: For $t=a^2b$ we have $tg^4=a^2ba^{12}b=a^{14}b^2$ and $gt=a^3ba^2b=a^{41}b^2=a^{14}b^2$, so $t\adhit g^4=g$. Note $t\inv=a^7b^2$. To calculate $t\adhit\chi_{uv}$ we need $t\inv\adhit a^3=a^3$ and $t\inv\adhit b=a^7b^3a^2b=a^9b$. Thus $t\adhit\chi_{uv}(a^3)=\chi_{uv}(a^3)=\zeta^u$ and $t\adhit\chi_{uv}(b)=\chi_{uv}(a^9b)=\zeta^{3u+3v}$. In other words $t\adhit\chi_{uv}=\chi_{u,v+u}$. The following diagram lists the $27$ characters of the centralizer. Characters connected by a dashed, resp.\ solid line have their images under $\IndtoYD_g$ mapped to each other by $\auta_2$, resp. $\autb_4$. To avoid overloading the diagram, we have not indicated the direction of the actions, nor closed the cycles of a repeated application of the two mappings.
      \newcommand\beide{\ar@<.3ex>@{|-->}[r]\ar@<-.3ex>@{|->}[r]}
      \newcommand\solidd{\ar@{-}[d]}
      \newcommand\dashr{\ar@{--}[r]}
  \begin{equation*}
    \xymatrix{%
      \chi_{00}&&\chi_{01}\dashr&\chi_{02}\\
      \chi_{10}\solidd\dashr&\chi_{20}\solidd\dashr&\chi_{40}\solidd\dashr&\chi_{80}\solidd\dashr&\chi_{70}\solidd\dashr&\chi_{50}\solidd\\
      \chi_{11}\solidd\dashr&\chi_{22}\solidd\dashr&\chi_{41}\solidd\dashr&\chi_{82}\solidd\dashr&\chi_{71}\solidd\dashr&\chi_{52}\solidd\\
      \chi_{12}\dashr&\chi_{21}\dashr&\chi_{42}\dashr&\chi_{81}\dashr&\chi_{72}\dashr&\chi_{51}\\
      \chi_{30}\dashr&\chi_{60}&\chi_{31}\dashr&\chi_{62}&\chi_{32}\dashr&\chi_{61}
    }
  \end{equation*}
  Thus, the $54$ simples associated to $g$ and $g^2$ fall in $6$ orbits under the combined actions of $\auta_{2}$ and $\autb_{2}$; the elements in each orbit are I-equivalent. 

  The result for $g=b$ is substantially different, although $C_G(b)=\langle a^3,b\rangle$ as well. Since $b^4=b$, the action of $\autb_4$ on the image of $\IndtoYD_b$ is trivial. Thus, the $54$ simples associated to $b$ and $b^2$ fall into $8$ orbits.
\end{Expl}

\begin{Rem}
  In the three examples studied in this section, one can form an autoequivalence $\doubleadams_r$ acting within the image of one of the functors $\IndtoYD_g$. It so happens that in each of the examples this autoequivalence is induced by an automorphism of the group $G$ under consideration. In the first example this is the automorphism fixing $a$ and $\tau$ and inverting $b$. In the second example it is the automorphism fixing $a$ and $\tau$ while switching $b$ and $c$. In the third example the automorphism of $G$ sending $a$ to $a^{22}$ and $b$ to $a^{18}b$ can be checked to (exist and) send $a^3b$ to itself, $a^3$ to $a^{12}=(a^3)^4$ and $b$ to $a^{18}b=(a^3)^6b$, so that its effect on $\IndtoYD(\chi_{uv})$ is to map it to $\IndtoYD(\chi_{4u,v+2u})$, the same as $\autb_{4}\inv$ followed by $\auta_4$.

  In particular, all the I-equivalences of characters in these cases that are accounted for by the joint operations of $\auta$- and $\autb$-equivalences can also be accounted for by $\auta$-equivalences and group automorphisms, thus leaving out the somewhat more mysterious $\autb$-equivalences.

  We have remarked in the previous section that in general the $\doubleadams_r$-autoequivalences are not in general monoidal, so cannot, \emph{a fortiori}, come from group automorphisms. We haven't, however, worked out an example of this generality.
\end{Rem}

\section{Doubles of symmetric groups}
\label{sec:doubl-symm-groups}

One of the author's original motivations for the study of the functors $\autb_r$ was to explain an observation made by Courter \cite{MR3103664} during a computer-aided investigation into the higher Frobenius-Schur indicators of the simple modules of Drinfeld doubles of symmetric groups. To wit, many of these simple modules fall into rather large I-equivalence classes; for larger ranks of the symmetric groups each I-equivalence classes is contained in the image of one of the functors $\IndtoYD_g$. We will see in this section that the functors $\autb_r$ \emph{do not help in the least} to explain the abundance of I-equivalent modules.

The fact that Galois conjugate characters of $C_G(g)$ yield I-equivalent representations of $D(G)$ is (albeit not mentioned in \cite{MR3103664}) quite obvious, but accounts for only relatively few of the I-equivalent representations. To discuss a few examples, we use the notations of \cite{MR3103664} as well as the character tables of centralizers collected there in the appendix.

Thus, the characters of $C_{S_5}(u_i)$ are integer valued for $i=1,2,3$, but $\chi_{1.3}$ and $\chi_{1.5}$ are I-equivalent, and the six characters associated to $u_2$ fall in only two I-equivalence classes, as do the five characters associated to $u_3$. Things look slightly better for $u_4$, where $\eta_{4.3}=\adams_{-1}(\eta_{4.4})$ and $\eta_{4.5}=\adams_{-1}(\eta_{4.6})$; this does not account, hovever, for $\chi_{4.1}$ falling in the same I-equivalence class with $\chi_{4.3}$ and $\chi_{4.4}$, nor for $\chi_{4.2}$ falling in the I-equivalence class of the two other simples. Worse, $C_{S_5}(u_5)$ has the same character table as $C_{S_5}(u_4)$, but here all the six associated simples are I-equivalent. For $u_6$ we get four characters in three Galois conjugacy classes, but only one I-equivalence class, and for $u_7$ five characters in two conjugacy classes, and again only one I-equivalence class.

Summing up, Galois conjugacy, or the operations $\auta_r$, account for some, but by far not all I-equivalences. We have seen in the previous section that the operation $\autb_r$ can produce I-equivalences not accounted for by $\auta_r$, but unfortunately in the present case at least computer experiments suggest that this phenomenon does not occur at all.

\begin{Conj}
  For any $r$ relatively prime to $\exp(S_n)$ and $V\in\lYD{S_n}$ we have $\auta_r(V)=\autb_r(V)$.
\end{Conj}

Note that for any $g\in S_n$ and $r$ relatively prime to $\exp(S_n)$, the power $g^r$ is conjugate to $g$ since it has the same cycle structure. Thus $\autb_r(\IndtoYD_g(\Rep(C_{S_n}(g))))=\IndtoYD_g(\Rep(C_{S_n}(g)))$. Assume $g^r=t\adhit g$ for $t\in S_n$. Then $\autb_r$ and $\auta_r$ agree on the image of $\IndtoYD_g$ if and only if $\chi(t\adhit x)=\chi(x^r)$ for every $x\in C_{S_n}(g)$, that is, if and only if $t\adhit x$ and $x^r$ are conjugate in $C_{S_n}(g)$.

Since $\auta_{rs}=\auta_r\auta_s$ and $\autb_{rs}=\autb_r\autb_s$, it suffices to verify this for $r$ in a set of generators for the prime residues modulo $\exp(S_n)$. The following piece of GAP \cite{GAP4} code defines a function comparing the actions of $\auta_r$ and $\autb_r$ for all $r$ in such a set of generators, for the case of a symmetric group $S_n$ of given degree $n$. It uselessly prints an unstructured list of permutations and numbers to assure the user of the progress it is making, but its main purpose is to stop and alert the user, should ever $g$ be found and $x\in C_{S_n}(g)$ such that $g^r=t\adhit g$ but $x^r$ not conjugate to $t\adhit x$ in $C_{S_n}(g)$. We have tested the symmetric groups up to \bishierher; for larger degrees GAP reported memory problems, and no further attempts were made to circumvent these or use a bigger machine.

\lstset{language=GAP}
\begin{figure}\caption{GAP code to check $\auta_r=\autb_r$ on $\lYD{S_n}$}\label{fig:GAP}
\begin{lstlisting}
compare:=function(degree)
  local group,exponent,rgenerators,rlist,classes,
        class,g,c,cclasses,r,t,cclass,x;
  group:=SymmetricGroup(degree);
  exponent:=Exponent(group);
  rgenerators:=GeneratorsPrimeResidues(exponent);
  rlist:=Flat(rgenerators.generators);
  classes:=ConjugacyClasses(group);
  for class in classes do
    g:=Representative(class);
    Print(g,",");
    c:=Centralizer(group,g);
    cclasses:=ConjugacyClasses(c);
    for r in rlist do
      Print(r,",");
      t:=RepresentativeAction(group,g,g^r);
      for cclass in cclasses do
        x:=Representative(cclass);
        if not IsConjugate(c,x^t,x^r) 
           then Print("FAILS");break;fi;
      od;
    od;
  od;
end;
\end{lstlisting}
\end{figure}

\clearpage
\printbibliography

\end{document}